\documentclass[12pt]{amsart}
\usepackage{amssymb}
\usepackage{amsmath, amscd}
\usepackage{amsthm}
\usepackage{amstext}
\usepackage[table,xcdraw]{xcolor}
\usepackage[colorlinks=true, linkcolor=blue,urlcolor=blue]{hyperref}

\newtheorem{theorem}{Theorem}[section]

\newtheorem{proposition}[theorem]{Proposition}
\newtheorem{lemma}[theorem]{Lemma}

\newtheorem{corollary}[theorem]{Corollary}
\theoremstyle{definition}

\newtheorem{definition}[theorem]{Definition}

\newtheorem{remark}[theorem]{Remark}

\begin{document}

\author[A. Moussavi]{Ahmad Moussavi}
\address{Department of Mathematics, Tarbiat Modares University, 14115-111 Tehran Jalal AleAhmad Nasr, Iran}
\email{moussavi.a@modares.ac.ir; moussavi.a@gmail.com}

\author[P. Danchev]{Peter Danchev}
\address{Institute of Mathematics and Informatics, Bulgarian Academy of Sciences, 1113 Sofia, Bulgaria}
\email{danchev@math.bas.bg; pvdanchev@yahoo.com}

\author[A. Javan]{Arash Javan}
\address{Department of Mathematics, Tarbiat Modares University, 14115-111 Tehran Jalal AleAhmad Nasr, Iran}
\email{a.darajavan@modares.ac.ir; a.darajavan@gmail.com}

\author[O. Hasanzadeh]{Omid Hasanzadeh}
\address{Department of Mathematics, Tarbiat Modares University, 14115-111 Tehran Jalal AleAhmad Nasr, Iran}
\email{o.hasanzade@modares.ac.ir; hasanzadeomiid@gmail.com}

\title{A New Characterization of Semi-Tripotent Rings}

\maketitle




\begin{abstract} We give a comprehensive study of the so-called {\it semi-tripotent rings} obtaining their new and non-trivial characterization as well as a complete description in terms of sums and products of some special elements. Particularly, we explore in-depth when a group ring is semi-tripotent. Our results somewhat supply those established by Ko\c{s}an et al. in Can. Math. Bull. (2019).
\end{abstract}

\section{Introduction and Fundamentals}

Throughout this paper, all rings are assumed to be unital and associative. Almost all symbols and terminology are standard being consistent with the well-known books \cite{lamf} and \cite{lame}. Thus, the Jacobson radical, the lower nil-radical, the set of nilpotent elements, the set of idempotent elements, the set of tripotent elements and the set of units of \( R \) are, respectively, denoted by \( J(R) \), \( \text{Nil}_{*}(R) \), \( \text{Nil}(R) \), \( \text{Id}(R) \), \( \text{Tr}(R) \) and \( U(R) \).

The core focus of the present investigation is the set
\begin{align*}
  J(R) \subseteq \Delta(R) &= \{ x \in R : x + u \in U(R) \text{ for all } u \in U(R) \} \\
                           &= \{ x \in R : 1 - xu \text{ is invertible for all } u \in U(R) \} \\
                           &= \{ x \in R : 1 - ux \text{ is invertible for all } u \in U(R) \},
\end{align*}
which was examined by Lam \cite[Exercise 4.24]{lame} and more recently by Leroy-Matczuk \cite{lmr}. The authors in \cite[Theorems 3 and 6]{lmr} indicate that \( \Delta(R) \) represents the largest Jacobson radical subring of \( R \) that remains closed under multiplication by all units (or even by all quasi-invertible elements) of \( R \).

Historically, in ring theory, strongly nil-clean rings possess a significant importance: a ring \(R\) is called \textit{strongly nil-clean} if every element of \(R\) can be expressed as the sum of an idempotent and a nilpotent element that commute with each other (see \cite{chenb, diesl, kzw}). These rings were completely classified as rings which are boolean modulo their nil Jacobson radical (see, e.g., \cite{dl} and \cite{kwz}).

Furthermore, Chen and Sheibani in \cite{css} generalized this concept and introduced strongly 2-nil-clean rings: a ring \(R\) is said to be \textit{strongly 2-nil-clean} if every element of \(R\) can be written as the sum of a tripotent element (that is, an element \(x\) such that \(x^3 = x\)) and a nilpotent element that commute with each other.

On the other hand, \textit{strongly J-clean} rings are those rings in which every element can be written as the sum of an idempotent and an element from the Jacobson radical \(J(R)\) that commute with each other (see \cite{csj, chensjc}). A relevant version of this is considered in \cite{cui} by defining so-named {\it 2-UJ rings} that are those rings for which the square of each unit is a sum of an idempotent and an element from the Jacobson radical.

Later on, combining these two notions, Ko\c{s}an et al. in \cite{kyzr} defined the so-termed \textit{semi-tripotent} rings in which each element is the sum of a tripotent element and an element from \(J(R)\).

Inspired by all of this, we call a ring $R$ to be {\it $\Delta$-tripotent}, or just a {\it DT ring} for short, provided every element of $R$ is a sum of an element from $\Delta(R)$ and of element from $\text{Tr}(R)$, i.e., $$R=\Delta(R)+\text{Tr}(R).$$

\noindent We are motivating to demonstrate in the next section the curious non-elementary equivalence between the classes of \textit{DT rings} and \textit{semi-tripotent rings} (see Corollary~\ref{gdt iff semi tripotent}). Although the structure of semi-tripotent group rings has been elegantly described in \cite{kyzr}, our approach differs from the methods used there. For this: (1) we are focusing on element-based proofs (rather than the ring-theoretic structures); (2) we are providing more optimized formulations of certain structures; (3) we are basing all arguments on the properties of $\Delta(RG)$ -- in fact, given that $\Delta(R)$ is a relatively new set with unexplored in detail structures (unlike $J(R)$), this section examines key features of $\Delta(RG)$ and refines existing frameworks.

\medskip

Our principal results are planned to be proved in Theorems~\ref{p in Delta and p-group}, \ref{p neq 2 and 3}, \ref{G is 2-group}, \ref{G is 3-group or elementary 2-group} and \ref{2 in Delta RG is GDT} as well as Theorems~\ref{boolean and yaqub}, \ref{avali T GDT}, \ref{maj3} and \ref{maj4}, respectively.

\section{Basic Properties}

We begin here with the following preliminaries. The first tool is our key notion as formulated above.

\begin{definition}
A ring $R$ is called {\it $\Delta$-tripotent} or just a \textit{DT ring} for short if, for every $r \in R$, there exist $e \in \operatorname{Tr}(R)$ and $d \in \Delta(R)$ such that $r = e + d$.
\end{definition}

The following lemma can easily be proven, so we leave it to the interested reader.

\begin{lemma}\label{lemma 0} The following two assertions are true:
(1) Suppose \( R = \prod_{i \in I} R_i \). Then, \( R \) is a DT ring if, and only if, for each \( i \in I \), \( R_i \) is a DT ring.

(2) Suppose \( R \) is a ring and \( I \) is an ideal of \( R \) such that \( I \subseteq J(R) \). Then, \( R/I \) is a DT ring.
\end{lemma}

The next two statements are pivotal.

\begin{lemma}\label{prop d,e}
Let $R$ be a ring. Then, the following three conditions hold:
\begin{enumerate}
    \item $(f \pm f^2)d$ and $d(f \pm f^2) \in \Delta(R)$ for every $f \in \operatorname{Tr}(R)$ and $d \in \Delta(R)$.
    \item $2ed \in \Delta(R)$ and $2de \in \Delta(R)$ for every $e \in \operatorname{Id}(R)$ and $d \in \Delta(R)$.
    \item $2fd \in \Delta(R)$ and $2df \in \Delta(R)$ for every $f \in \operatorname{Tr}(R)$ and $d \in \Delta(R)$.
\end{enumerate}
\end{lemma}

\begin{proof}
We prove only (1), because (2) and (3) follow from (1). So, assuming $f \in \operatorname{Tr}(R)$ and $d \in \Delta(R)$, we write
\[
((1 - f^2) - f)((1 - f^2) - f) = 1 = ((1 - f^2) + f)((1 - f^2) + f),
\]
so, in view of \cite[Lemma 1(2)]{lmr}, we have that all of the elements $((1 - f^2) - f)d$, $d((1 - f^2) - f)$, $((1 - f^2) + f)d$ and $d((1 - f^2) + f)$ are in $\Delta(R)$. Since $\Delta(R)$ is known to be closed under addition, it follows at once that both $(f \pm f^2)d$ and $d(f \pm f^2)$ lie in $\Delta(R)$, as required.
\end{proof}

\begin{lemma}\label{power a then a}
Let $R$ be a DT ring such that $a^2 \in \Delta(R)$. Then, $a \in \Delta(R)$.
\end{lemma}

\begin{proof}
Since $a^2 \in \Delta(R)$, we have $1 - a^2 \in U(R)$. On the other side, since $(1 + a)(1 - a) = 1 - a^2 \in U(R)$, it follows that $1 + a \in U(R)$. Therefore, \cite[Lemma 1(2)]{lmr} employs to find that $(1 + a)d, d(1 + a) \in \Delta(R)$.

Assume now that $a = e + d$ is a DT representation. Thus, we write:
\[
ae = a^2 + d - (a + 1)d \in \Delta(R),
\]
\[
ea = a^2 + d - d(a + 1) \in \Delta(R).
\]
But, since $a = e + d$, we deduce $d^2 = (a - e)^2 = a^2 - ae - ea + e$ implying that
\[
e = d^2 - a^2 + ae + ea \in \Delta(R).
\]
Consequently, $e^2 \in \Delta(R) \cap \operatorname{Id}(R) = \{0\}$, and hence
\[
e = e^2 e = 0.
\]
This shows that $a = d \in \Delta(R)$, as needed.
\end{proof}

As an immediate consequence, we yield.

\begin{corollary}
Let $R$ be a DT ring. Then, $\operatorname{Nil}(R) \subseteq \Delta(R)$.
\end{corollary}

We proceed by proving the following claims that are our basic instruments.

\begin{lemma}
Let $R$ be a DT ring. Then, the following two conditions are valid:
\begin{enumerate}
    \item $er - re \in \Delta(R)$ for every $e \in \operatorname{Id}(R)$ and $r \in R$.
    \item $fd \pm df \in \Delta(R)$ for every $f \in \operatorname{Tr}(R)$ and $d \in \Delta(R)$.
\end{enumerate}
\end{lemma}

\begin{proof}
(1) Let $e \in \operatorname{Id}(R)$ and $r \in R$. Then, we have:
\[
[er(1-e)]^2 = 0 = [(1-e)re]^2.
\]
Exploiting Lemma \ref{power a then a}, we obtain:
\[
er(1-e) \in \Delta(R) \implies er - ere \in \Delta(R),
\]
\[
(1-e)re \in \Delta(R) \implies re - ere \in \Delta(R).
\]
Since $\Delta(R)$ is known to be closed under addition, we derive:
\[
er - re = (er - ere) - (re - ere) \in \Delta(R).
\]

(2) Let $f \in \operatorname{Tr}(R)$ and $d \in \Delta(R)$. Since $f^2 \in \operatorname{Id}(R)$, owing to (1), we write $f^2d - df^2 \in \Delta(R)$. Thus, with the aid of Lemma \ref{prop d,e}(1), we arrive at:
\[
fd - df = fd + f^2d - f^2d - df - df^2 + df^2 = (f + f^2)d + (df^2 - f^2d) - d(f + f^2) \in \Delta(R).
\]
Similarly, we obtain:
\[
fd + df = fd + f^2d - f^2d + df - df^2 + df^2 = (f + f^2)d + (df^2 - f^2d) - d(f - f^2) \in \Delta(R),
\]
as expected.
\end{proof}

\begin{proposition}\label{6 in Delta}
Let $R$ be a DT ring. Then, $6 \in \Delta(R)$.
\end{proposition}

\begin{proof}
Letting $2 = e + d$ be a DT representation, we can write $4 = 2^2 = e^2 + 2ed + d^2$. Set $d_1 := 2ed + d^2$. Consulting with Lemma \ref{prop d,e}, we have $d_1 \in \Delta(R)$. Therefore, it must be that
\[
2 = 4 - 2 = (e^2 - e) + (d_1 - d).
\]
Put $d_2 := d_1 - d$. Since $\Delta(R)$ is closed under addition, we get $d_2 \in \Delta(R)$. So, we receive that
\[
4 = 2^2 = 2(e^2 - e) + 2(e^2 - e)d_2 + d_2^2.
\]
Again viewing Lemma \ref{prop d,e}, we have $d_3 = 2(e^2 - e)d_2 + d_2^2 \in \Delta(R)$. Moreover, since $2 = e + d$, we have:
\[
4 = 2(e^2 - e) + d_3 = (e + d)(e^2 - e) + d_3 = (e - e^2) + d(e^2 - e) + d_3.
\]
Thus, $d_4 = d(e^2 - e) + d_3 \in \Delta(R)$ and, therefore,
\[
6 = 4 + 2 = [(e - e^2) + d_4] + [(e^2 - e) + d_2] = d_4 + d_2 \in \Delta(R),
\]
as wanted.
\end{proof}

The next consequence is very useful for our successful presentation.

\begin{corollary}\label{6 in Jacobson}
Let $R$ be a DT ring. Then, $6 \in J(R)$.
\end{corollary}

\begin{proof}
First, we show that $12 \in J(R)$. To that end, choose $r \in R$ to be arbitrary with a DT representation $r = e + d$. Since thanks to Proposition~\ref{6 in Delta}, we extract $6 \in \Delta(R)$, one observes that Lemma \ref{prop d,e} applies to write that
\[
1 - 12r = 1 - 12(e + d) = 1 - 12e - 12d = 1 - 2 \cdot 6e - 12d \in 1 + \Delta(R) + \Delta(R) \subseteq U(R),
\]
which insures $12 \in J(R)$. Moreover, since $36 = 3 \cdot 12 \in J(R)$, for every $s \in R$ we detect:
\[
(1 - 6s)(1 + 6s) = 1 - 36s^2 \in U(R),
\]
which discovers that $6 \in J(R)$, as desired.
\end{proof}

Two more consequences now hold:

\begin{corollary}\label{2 in U(R) iff 3 in J(R)}
Let $R$ be a DT ring. Then:

(1) $2 \in U(R)$ if, and only if, $3 \in J(R)$.

(2) $3 \in U(R)$ if, and only if, $2 \in J(R)$.
\end{corollary}

\begin{proof}
It is straightforward based on Corollary \ref{6 in Jacobson}.
\end{proof}

\begin{corollary}\label{p and Delta p=2 or 3}
Let $R$ be a DT ring such that, for some prime number $p$, we have $p \in \Delta(R)$. Then, either $p=2$ or $p=3$.
\end{corollary}

\begin{proof}
Utilizing Proposition \ref{6 in Delta}, we get $6 \in \Delta(R)$. If $p \neq 2$ and $p \neq 3$, then $(p,6)=1$, which ensures $1 \in \Delta(R)$. This, however, is a contradiction. That is why, we must have either $p=2$ or $p=3$, as stated.
\end{proof}

\section{Group Rings}

Let $R$ be a ring and $G$ a group. As usual, the notation $RG$ stands for the group ring as being a module over $R$ with elements of $G$ as a basis. The homomorphism $\varepsilon \colon RG \to R$, defined by $\sum r_g g \mapsto \sum r_g$, is standardly known as the \textit{augmentation homomorphism} of $RG$. Its kernel, $\ker(\varepsilon)$, referred to as the \textit{augmentation ideal} of $RG$, is denoted by $\varepsilon(RG)$, and equals to $$\varepsilon(RG)=\left\{ \sum_{g \in G}a_g(1 - g) \colon 1 \neq g \in G, a_g \in G \right\}.$$

Traditionally, a group $G$ is called \textit{locally finite}, provided that any subgroup generated by a finite subset of $G$ is itself finite. When $p$ is a prime, a $p$-group means that every its element has order equal to a power of $p$. If all non-identity elements of a group have order exactly $p$, the group is said to have \textit{exponent} $p$. The notation $C_n$ represents the classical cyclic group having only $n$ elements.

\medskip

We now come to our first general result in this section.

\begin{theorem}\label{p in Delta and p-group}
Let $R$ be a ring and $G$ a group such that, for each $1 \neq g \in G$, we have $1-g \in \Delta(RG)$. Then, $G$ is a $p$-group, where $p \in \Delta(R)$.
\end{theorem}

\begin{proof}
First, we show that $G$ is a torsion group. To this goal, suppose the contrary that there exists $g \in G$ of infinite order. Since $1-g \in \Delta(RG)$, we have $1-g+g^2 \in U(RG)$. Therefore, there exist integers $n < m$ and elements $a_i$ with $a_n \neq 0 \neq a_m$ such that
\[
(1-g+g^2)\sum_{i=n}^{m} a_ig^i = 1.
\]
This unambiguously leads to a contradiction, and thus every element $g \in G$ must have finite order. (Note that in view of \cite[Proposition 4(i)]{coon}, for every subgroup $H$ of $G$, it must be that $U(RG) \cap RH \subseteq U(RH)$.)

Furthermore, since each element \( g \in G \) has finite order, let \( n \) be the order of some \( g \in G \), and let \( p \) be a prime divisor of \( n \). Then, there will exist an element \( g \in G \) of order \( p \). But, for any \( u \in U(R) \), the element
\[ 1 + (1 - g)(p + (p-1)g + (p-2)g^2 + \dots + g^{p-1})u \]
admits a right inverse \( x \). Same as in the previous arguments, \( x \) can be expressed as a polynomial of \( g \), and hence without loss of generality the equation
\[ \big(1 + (1 - g)(p + (p-1)g + (p-2)g^2 + \dots + g^{p-1})u)\big)x=1 \]
can be interpreted to hold in the ring \( RG = RC_p \), where \( C_p \) is the cyclic group of order \( p \).

Now, let \( J \) be the ideal of \( R \) generated by the elements \( \sum_{i=0}^{p-1}g^{i} \). So, assuming $\overline{RG} := RG/J$, we conclude $\overline{1+pu} = \bar{1}$, where
\[ x = a_0 + a_1g + \cdots + a_{p-2}g^{p-2}. \]
Thus, we obtain $(1+pu)a_0 = 1$, and since \( u \in U(R) \) was arbitrary, this guarantees $p \in \Delta(R)$.

We next assert that no other prime \( q \in \Delta(R) \) can satisfy this property, whence $(p,q)=1$ assuring that $1 \in \Delta(R)$. This is an obvious contradiction, however. Thus, \( p \) is the unique prime dividing the order of any element in \( G \), and hence \( G \) must be a \( p \)-group, as asserted.
\end{proof}

We now need a series of some technical things.

\begin{lemma}\label{1}
The inclusion $\varepsilon(\Delta(RG)) \subseteq \Delta(R)$ is always fulfilled.
\end{lemma}

\begin{proof}
Choose $d \in \Delta(RG)$ and $u \in U(R)$. Since $\varepsilon(u)=u$ and $\varepsilon(U(RG)) \subseteq U(R)$, we readily inspect that
\[
1 - u\varepsilon(d) = \varepsilon(1 - ud) \in \varepsilon(U(RG)) \subseteq U(R),
\]
as asked for.
\end{proof}

\begin{lemma}\label{Delta(RG) subset Delta(RH)}
Let $R$ be a ring and $H$ a subgroup of $G$. Then, the inclusion $\Delta(RG) \cap RH \subseteq \Delta(RH)$ holds.
\end{lemma}

\begin{proof}
Choosing $f \in \Delta(RG) \cap RH$ and $u \in U(RH) \subseteq U(RG)$, we apply \cite[Proposition 4(i)]{coon} to get that $U(RG) \cap RH \subseteq U(RH)$ and, therefore,
\[
1 - fu \in U(RG) \cap RH \subseteq U(RH),
\]
forcing $f \in \Delta(RH)$, as pursued.
\end{proof}

\begin{lemma}\label{u in U(RG)}
Let $R$ be a ring and $G$ a group with $\varepsilon(RG) \subseteq J(RG)$. Then, for any $u \in RG$ with $\varepsilon(u) \in U(R)$, we have $u \in U(RG)$.
\end{lemma}

\begin{proof}
Since $\varepsilon(u) \in U(R)$, there is $r \in R$ such that $\varepsilon(u)r = r\varepsilon(u) = 1$. But, because $\varepsilon$ is surjective, there is $v \in RG$ with $\varepsilon(v) = r$. Thus, one checks that
\[
\varepsilon(1-uv) = 1 - \varepsilon(u)\varepsilon(v) = 1 - \varepsilon(u)r = 1-1 = 0,
\]
and so $1-uv \in \ker\varepsilon \subseteq J(RG)$ giving $uv \in U(RG)$. Similarly, one obtains that $vu \in U(RG)$, and hence $u \in U(RG)$, as claimed.
\end{proof}

\begin{lemma}\label{Delta(R) subset Delta(RG)}
Let $R$ be a ring with $p \in J(R)$ for some prime $p$, and let $G$ be a locally finite $p$-group. Then, $\Delta(R)G \subseteq \Delta(RG)$.
\end{lemma}

\begin{proof}
Using \cite[Lemma 2]{zhoucl}, we have $\varepsilon(RG) \subseteq J(RG)$. Now, let $d = \sum_{g \in G} a_g g \in \Delta(R)G$ and $u \in U(RG)$. Since $\Delta(R)$ is a subring and thus it is closed under addition, it must be that $\sum a_g \in \Delta(R)$. As $\varepsilon(u) \in U(R)$, we may write
\[
\varepsilon(1-ud) = 1 - \varepsilon(u)\sum a_g \in U(R).
\]
Employing Lemma \ref{u in U(RG)}, we have $1-ud \in U(RG)$, which shows $d \in \Delta(RG)$, as required.
\end{proof}

\begin{lemma}\label{RG gdt then R is so}
Let $R$ be a ring and $G$ a group. If $RG$ is a DT ring, then $R$ is also a DT ring.
\end{lemma}

\begin{proof}
Let $a \in R$. Since $RG$ is a DT ring, suppose $a = e + d$ is a DT representation in $RG$. As $a \in R$, we write
\[
a = \varepsilon(a) = \varepsilon(e) + \varepsilon(d).
\]
Clearly, $\varepsilon(e) \in \operatorname{Tr}(R)$, and hence Lemma~\ref{1} enables us that $\varepsilon(d) \subseteq \Delta(R)$, as needed.
\end{proof}

Referring to Lemma \ref{RG gdt then R is so}, if $RG$ is a DT ring, then $R$ is too a DT ring. In what follows, we attempt to menage the conditions that the group $G$ must satisfy when $RG$ is a DT ring.

\medskip

Specifically, the following statements are true.

\begin{lemma}\label{1-g2 in Delta}
Let $RG$ be a DT ring. Then, for every $g \in G$, we have $1-g^2 \in \Delta(RG)$.
\end{lemma}

\begin{proof}
Supposing $g = e + d$ is a DT representation, we then write $$e = g - d \in U(RG) + \Delta(RG) \subseteq U(RG).$$ Hence, $$e^2 \in \operatorname{Id}(RG) \cap U(RG) = \{1\}.$$ Therefore, invoking Lemma \ref{prop d,e}, we deduce
\[
g^2 = 1 + (ed + de) + d^2 \in 1 + \Delta(RG).
\]
Thus, for each $g \in G$, one sees that $1 - g^2 \in \Delta(RG)$, as expected.
\end{proof}

\begin{lemma}\label{G torsin}
Let $RG$ be a DT ring. Then, $G$ is a torsion group.
\end{lemma}

\begin{proof}
A consultation with Lemma \ref{1-g2 in Delta} gives that, for every $g \in G$, $1 - g^2 \in \Delta(RG)$. Therefore, $1 - g + g^2 \in U(RG)$. Consequently, there are integers $n < m$ and elements $a_i$ with $a_n \neq 0 \neq a_m$ such that
\[
(1 - g + g^2)\sum_{i=n}^{m} a_ig^i = 1.
\]
This obviously leads us to a contradiction, and so each element $g \in G$ has finite order, as wanted.
\end{proof}

We are now ready to establish a series of our further major assertions.

\begin{theorem}\label{p neq 2 and 3}
Let $RG$ be a DT ring with $2 \notin \Delta(R)$ and $3 \notin \Delta(R)$. Then, $G$ is an elementary 2-group.
\end{theorem}

\begin{proof}
First, we prove that $G$ is a 2-group. To this target, suppose there is $g \in G$ of odd order $o(g) = 2k+1$. In virtue of Lemmas \ref{1-g2 in Delta} and \ref{Delta(RG) subset Delta(RH)}, we have
\[
1 - g = -g(1 - g^{2k}) \in \Delta(RG) \cap R\langle g \rangle \subseteq \Delta(R\langle g \rangle).
\]
Moreover, for every $1 \leq i \leq 2k+1$, we write
\[
1 - g^i = (1 - g)(1 + g + \cdots + g^{i-1}) \in \Delta(R\langle g \rangle).
\]
Now, Theorem \ref{p in Delta and p-group} is a guarantor that the cyclic subgroup $\langle g \rangle$ is a $p$-group with $p \in \Delta(R)$. However, Lemma \ref{p and Delta p=2 or 3} assures that we must have either $p=2$ or $p=3$, which manifestly contradicts our assumption. Therefore, any element $g \in G$ has even finite order.

If, however, there is $g \in G$ with $o(g) = p_1^{\alpha_1}\cdots p_n^{\alpha_n}$, then, for each $1 \leq i \leq n$, we will arrive at $p_i = 2$, as for otherwise $G$ would contain an element of odd order, which by what we have already shown above is impossible. Thus, $G$ is indeed a 2-group.

Now, let $k>1$ be the smallest positive integer such that $g^{2^k} = 1$. Consider the subgroup
\[
H = \{1, g^2, g^4,g^6 \ldots, g^{2^{k-1}}\}.
\]
Certainly, $H$ is a subgroup of $G$. An exploitation of Lemmas \ref{Delta(RG) subset Delta(RH)} and \ref{1-g2 in Delta} shows that, for every $h \in H \subseteq G$,
\[
1 - h \in \Delta(RG) \cap RH \subseteq \Delta(RH).
\]
Thus, with Theorem \ref{p in Delta and p-group} at hand, $H$ is a 2-group with $2 \in \Delta(R)$, contradicting our assumption. Consequently, we conclude $k=1$, yielding that $G$ is really an elementary 2-group, as formulated.
\end{proof}

\begin{theorem}\label{G is 2-group}
Let $RG$ be a DT ring with $2 \in \Delta(R)$. Then, $G$ is a 2-group.
\end{theorem}

\begin{proof}
First, we show that $2 \in J(RG)$. Since $2 \in \Delta(R)$, we have $3 \in U(R) \subseteq U(RG)$. Moreover, by virtue of Lemma \ref{6 in Jacobson}, we know that $6 \in J(RG)$, detecting that $2 = 6/3 \in J(RG)$.

But Lemma \ref{G torsin} tells us that $G$ is a torsion group, and Lemma \ref{1-g2 in Delta} informs us that, for every $g \in G$, $1 - g^{2k} \in \Delta(RG)$ is valid (note that $g^k \in G$). Furthermore, since $\Delta(RG)$ is a subring, we obtain:
\[
1 + g^{2k} = 1 - g^{2k} + 2g^{2k} \in \Delta(RG) + J(RG) \subseteq \Delta(RG).
\]

But, as $\Delta(RG)$ is closed with respect to multiplication by unit elements, we extract:
\[
\sum_{i=0}^{2k} g^i = \sum_{i=0}^{k-1} g^i(1 + g^{2(k-i)}) + g^k \in \Delta(RG) + U(RG) \subseteq U(RG).
\]

Now, suppose there is an element in $G$ with odd order $p$. Then, there is $g \in G$ with $o(g) = p$. Letting $p = 2k + 1$, then
\[
(1 - g)\sum_{i=0}^{2k} g^i = 0.
\]
Since $\sum_{i=0}^{2k} g^i \in U(RG)$, we must have $1 - g = 0$, i.e., $g=1$, which is wrong. Therefore, $G$ must be a 2-group, as stated.
\end{proof}

\begin{theorem}\label{G is 3-group or elementary 2-group}
Let $RG$ be a DT ring with $3 \in \Delta(R)$, and let $G$ be a $p$-group. Then, either $G$ is a 3-group or an elementary 2-group.
\end{theorem}

\begin{proof}
Analogously to Theorem \ref{G is 2-group}, we can establish that $3 \in J(RG)$. Now, we consider two different cases:

\medskip

\textbf{Case 1:} $p=2$. We demonstrate that, for each $g \in G$, $g^2=1$. In fact, since $G$ is a 2-group, suppose $k>1$ is the smallest positive integer such that $g^{2^k}=1$. As $2^{k-1}$ is even, there is $m \in \mathbb{N}$ with $2^{k-1}=2m$. With the help of Lemma \ref{1-g2 in Delta}, we have:
\[
1-g^{2^{k-1}} = 1-g^{2m} \in \Delta(RG).
\]
Moreover, since $3 \in J(RG)$, it must be that
\[
1+2g^{2^{k-1}} = 1-g^{2^{k-1}}+3g^{2^{k-1}} \in \Delta(RG) + \Delta(RG) \subseteq \Delta(RG).
\]
This ensures
\[
1+g^{2^{k-1}} = 1+2g^{2^{k-1}} -g^{2^{k-1}} \in \Delta(RG) + U(RG) \subseteq U(RG).
\]
However, $(1+g^{2^{k-1}})(1-g^{2^{k-1}}) = 1-g^{2^k} = 0$ and $1+g^{2^{k-1}} \in U(RG)$, so we must have $1-g^{2^{k-1}} = 0$, contradicting the minimality of $k$. Hence, $k=1$.

\medskip

\textbf{Case 2:} $p \neq 2$. Let $g^{p^k}=1$. Since $p^k$ is odd, there is $m \in \mathbb{N}$ such that $p^k=2m+1$. With the aid of Lemma \ref{1-g2 in Delta}, we write:
\[
1-g^{2m} \in \Delta(RG).
\]
Multiplying both sides by $g$ gives $1-g \in \Delta(RG)$. Then, Theorem \ref{p in Delta and p-group} allows us to infer that $G$ is a $q$-group with $q \in \Delta(R)$. However, since $3 \in \Delta(R)$ and $\Delta(R)$ cannot contain two distinct primes, we derive $q=3$. Consequently, $G$ must be a 3-group, as promised.
\end{proof}

Our next chief result and its valuable consequence are the following ones.

\begin{theorem}\label{2 in Delta RG is GDT}
Let $R$ be a DT ring with $2 \in \Delta(R)$, and let $G$ be a locally finite 2-group. Then, $RG$ is a DT ring.
\end{theorem}

\begin{proof}
Since $6 \in J(R)$ and, by assumption $2 \in J(R)$, it follows from \cite[Lemma 2]{zhoucl} that $\varepsilon(RG) \subseteq J(RG) \subseteq \Delta(RG)$.

However, for any $f = \sum_{g \in G} a_g g \in RG$, we can write
\[
f = -\sum_{g \in G} a_g(1-g) + \sum a_g \in \varepsilon(RG) + R \subseteq J(RG) + R \subseteq \Delta(RG) + R,
\]
so we may assume $RG = \Delta(RG) + R$.

Under validity of this assumption, let $f = \sum_{g \in G} a_g g \in RG$. Then, there are $d \in \Delta(RG)$ and $a \in R$ such that $f = d + a$. Since $R$ is a DT ring, suppose $a = e + d'$ is a DT representation in $R$. So, by Lemma \ref{Delta(R) subset Delta(RG)}, we have $d' \in \Delta(RG)$. Therefore,
\[
f = e + (d' + d)
\]
is a DT representation, and thus $RG$ must be a DT ring, as asserted.
\end{proof}

\begin{corollary}\label{3 in Delta RG is GDT}
Let $R$ be a DT ring with $3 \in \Delta(R)$, and let $G$ be a locally finite 3-group. Then, $RG$ is a DT ring.
\end{corollary}

\begin{proof}
The proof is quite similar to that of Theorem \ref{2 in Delta RG is GDT}, so we voluntarily drop off the arguments.
\end{proof}



\section{Structural Theorems}

We start here with the following preliminary technicalities. The first one is a rather suspected affirmation.

\medskip

Recall that, imitating \cite{nstr}, a ring $R$ is {\it clean}, provided $R={\rm Id}(R)+U(R)$.

\begin{lemma}\label{gdt is clean}
Let $R$ be a DT ring. Then, $R$ is a clean ring.
\end{lemma}

\begin{proof}
Assume $r \in R$ and $r = e + d$ is a DT representation. Since a simple check leads to $((1 - e^2) - e) \in U(R)$ and $U(R) + \Delta(R) \subseteq U(R)$, we have
\[
r = (1 - e^2) + \left[e - (1 - e^2) + d\right],
\]
which is obviously a clean representation for $r$. Thus, $R$ is a clean ring, as expected.
\end{proof}

\begin{lemma}\label{R/J is reduced}
Let $R$ be a DT ring. Then, the factor-ring $R/J(R)$ is reduced.
\end{lemma}

\begin{proof}
Suppose $a^2 \in J(R) \subseteq \Delta(R)$. Contacting with Lemma \ref{power a then a}, we have $a \in \Delta(R)$. Now, we establish that, for every $r \in R$, $ra - ar \in \Delta(R)$. To this purpose, let $r = e + d$ be a DT representation. Thus, owing to Lemma \ref{prop d,e}, we write:
\[
ra - ar = (e + d)a - a(e + d) = (ea - ae) + (da - ad) \in \Delta(R).
\]
Since $a^2 \in J(R)$, we deduce $1 - r^2a^2 \in U(R)$, which teaches us that $1 - ra^2r \in U(R)$. Therefore,
\[
(1 - ra)(1 + ar) = (ar - ra) + (1 - ra^2r) \in \Delta(R) + U(R) \subseteq U(R).
\]
Finally, $1 - ra \in U(R)$, showing that $a \in J(R)$, as required.
\end{proof}

The next two statements are our powerful machineries necessary for attacking below the pursued equivalence between DT rings and semi-tripotent rings.

\begin{proposition}\label{Delta is ideal when 2 in U(R)}
Let $R$ be an DT ring such that $2 \in U(R)$. Then, $\Delta(R)$ is an ideal. In particular, under these conditions, $\Delta(R) = J(R)$.
\end{proposition}

\begin{proof}
Since $\Delta(R)$ is closed under addition, it is sufficient to show that, for any $d \in \Delta(R)$ and $r \in R$, $rd, dr \in \Delta(R)$. To this aim, assume that $rd = e + b$ and $r = f + b'$ are two DT representations. So, according to Lemma \ref{prop d,e}, we have $2fd \in \Delta(R)$. But, as $2 \in U(R)$, it follows from \cite[Lemma 1(2)]{lmr} that $fd \in \Delta(R)$. Thus, we find that
$$rd = e + b = fd + b'd  \Longrightarrow  e = fd+ b'd - b \in \Delta(R).$$
So, $e^2 \in \Delta(R) \cap Id(R) = \{0\}$ giving that $e = 0$. Consequently, $rd = b \in \Delta(R)$.

Similarly, it can be shown that $dr \in \Delta(R)$, thus proving the claim.
\end{proof}

\begin{proposition}\label{R is sum idempotent and delta}
Let $R$ be a DT ring with $3 \in U(R)$. Then, for any $a \in R$, we have $a = f + b$, where $f = f^2 \in R$ and $b \in \Delta(R)$.
\end{proposition}

\begin{proof}
Choose $a \in R$ and suppose $a = e + d$ is a DT representation. Hence,
\[
a - e^2 = (e - e^2) + d.
\]
Since $3 \in U(R)$, in accordance with Lemma \ref{2 in U(R) iff 3 in J(R)}, $2 \in J(R)$. Therefore,
\[
(e - e^2)^2 = -2(e - e^2) \in J(R),
\]
and looking at Lemma \ref{power a then a}, we obtain $e - e^2 \in \Delta(R)$. Consequently, $a - e^2 \in \Delta(R)$, which concludes the assertion.
\end{proof}

A ring \(R\) is termed a {\it DI ring} if each element \(r \in R\) can be expressed as \(r = e + b\), where \(e\) is an idempotent and \(b \in \Delta(R)\). Thereby, all DI rings are certainly DT rings.

Recall also that a {\it \(\Delta U\) ring} is defined by the property \(1 + \Delta(R) = U(R)\) (see cf. \cite{kkqt}).

\begin{lemma}\label{GDI is DU}
All DI rings are \(\Delta U\).
\end{lemma}

\begin{proof}
For any unit \(u \in U(R)\) with DI decomposition \(u = e + d\), we observe that \[e = u - d \in U(R) + \Delta(R) \subseteq U(R) \cap Id(R) = \{1\},\] as required.
\end{proof}

Likewise, we recollect the following result.

\begin{lemma}\cite[Proposition 2.3]{kkqt} \label{sum unit}
A ring \(R\) is \(\Delta U\) if, and only if, \(U(R) + U(R) \subseteq \Delta(R)\) is tantamount to \(U(R) + U(R) = \Delta(R)\).
\end{lemma}

A ring \(R\) is called {\it uniquely clean} if every element has a unique expression as the sum of an idempotent and a unit \cite{nzu}.

\medskip

We are prepared to prove the following.

\begin{corollary}\label{cor 2}
For any ring \(R\), the following two conditions are equivalent:

(1) \(R\) is uniquely clean;

(2) \(R\) is DI with central idempotents.
\end{corollary}

\begin{proof}
\((1) \Rightarrow (2)\): Observe that the unique cleanness ensures that all idempotents are central (see, for instance, \cite[Lemma 4]{nzu}), and that the decomposition \(a = e + d\) with \(d \in J(R) \subseteq \Delta(R)\) is fulfilled (see, for example, \cite[Theorem 20]{nzu}).

\((2) \Rightarrow (1)\): Given \(a \in R\), and consider \(a+1 = e + d\). Then, \(a = e + (d-1)\), that is a clean representation.

Suppose now \(e + u = f + v\) are two clean representations. Thus, Lemma \ref{sum unit} shows that \(e - f = v - u \in \Delta(R)\). Since all idempotents are central, it follows that \(e - f = (e-f)^3\), and so \[(e-f)^2 \in \Delta(R) \cap Id(R) = \{0\}.\] Therefore, \[e - f = (e-f)^3 = (e-f)(e-f)^2 = 0,\] whence \(e = f\).
\end{proof}

A fundamental result in ring theory states that a ring is Boolean precisely when it can be represented as a sub-direct product of copies of $\mathbb{Z}_2$. In their work \cite{css}, Chen and Sheibani introduced the concept of a Yaqub ring, defining it as a sub-direct product of copies of $\mathbb{Z}_3$. They established the following characterization: a ring $R$ is Yaqub if, and only, if the element 3 is nilpotent in $R$ and $R$ satisfies the tripotent property.

\medskip

We are now in a position to proceed by proving the following main statement.

\begin{theorem}\label{boolean and yaqub}
Assume \(R\) is a DT ring. Then, \(R/J(R) \cong R_1 \times R_2\), where \(R_1\) is either zero or a non-zero Boolean ring, and \(R_2\) is either zero or a non-zero Yaqub ring.
\end{theorem}

\begin{proof}
Knowing Lemma \ref{6 in Jacobson}, we have \(6 \in J(R)\). Set \(\bar{R} := R/J(R)\). Next, with the Chinese Remainder Theorem in hand, we write \(\bar{R} \cong R_1 \times R_2\), where \(R_1 := \bar{R}/2\bar{R}\) and \(R_2 := \bar{R}/3\bar{R}\). Since \(R\) is a DT ring, Lemma \ref{lemma 0}(2) assures that \(\bar{R}\) is also a DT ring. Thus, again Lemma \ref{lemma 0}(1) illustrates that \(R_1\) is a DT ring. Furthermore, since \(2 = 0\) in \(R_1\), we get \(3 \in U(R_1)\). Hence, Proposition \ref{R is sum idempotent and delta} enables us that \(R_1\) is a DI ring. Moreover, Lemma \ref{R/J is reduced} gives that \(R_1\) is reduced, and thus all idempotents in \(R_1\) are central. Therefore, Corollary \ref{cor 2} is applicable to derive that \(R_1\) is a uniquely clean ring. Note that, because \(J(\bar{R}) = 0\), it must be that \(J(R_1) = 0\). Finally, a utilization of \cite[Theorem 19]{nzu} infers that \(R_1\) is a Boolean ring, as wanted.

On the other hand, since $3=0$ in $R_2$, we know $2 \in U(R_2)$. Therefore, with Proposition \ref{Delta is ideal when 2 in U(R)} in mind, we conclude that $\Delta(R_2)=J(R_2)$, which demonstrates that $R_2$ is a semi-tripotent ring. Next, taking into account \cite[Theorem 3.5]{kyzr}, we have the isomorphism $R_2/J(R_2)\cong S_1 \times S_2$, where $S_1$ is either zero or a Boolean ring, and $S_2$ is either zero or a Yaqub ring. Since $J(\bar{R})=0$, it follows that $J(R_2)=(0)$. Besides, because $3 \in U(R_2)$, we deduce that $S_1=(0)$ and, consequently, $R_2$ is a Yaqub ring, as desired.
\end{proof}

As a valuable consequence, we now extract the so-desired equivalence between the classes of DT rings and semi-tripotent rings, as noticed above. 

\begin{corollary}\label{gdt iff semi tripotent}
Let $R$ be a ring. Then, the following two conditions are equivalent:

(1) $R$ is a DT ring.

(2) $R$ is a semi-tripotent ring.
\end{corollary}

\begin{proof}
$(1) \Rightarrow (2)$: An application of Theorem \ref{boolean and yaqub} gives that, for every $a \in R$, the containment $a-a^3 \in J(R)$ holds. Moreover, Lemma \ref{gdt is clean} shows that all idempotents lift modulo $J(R)$. Therefore, a usage of \cite[Theorem 3.5]{kyzr} allows to detect that $R$ is a semi-tripotent ring.

$(2) \Rightarrow (1)$: This direction is immediate since we always have $J(R) \subseteq \Delta(R)$.
\end{proof}

We now continue with certain element-wise considerations.

\begin{theorem}\label{avali T GDT}
Let $R$ be a ring. Then, the following five conditions are equivalent:

(1) $R$ is a DT ring.

(2) $R$ is a semi-tripotent ring.

(3) For each $a \in R$, $a = e+f+j$, where $e,f \in \text{Id}(R)$, $ef=fe$ and $j \in J(R)$.

(4) For each $a \in R$, $a = e-f+j$, where $e,f \in \text{Id}(R)$, $ef=fe$ and $j \in J(R)$.

(5) For each $a \in R$, $a = e-f+j$, where $e,f \in \text{Id}(R)$, $ef=0=fe$ and $j \in J(R)$.
\end{theorem}

\begin{proof}
$(1) \Leftrightarrow (2)$: This follows directly from Corollary \ref{gdt iff semi tripotent}.

$(2) \Leftrightarrow (3)$: The equivalence is established in \cite[Theorem 3.7]{kyzr}.

$(3) \Rightarrow (4)$: For any $a \in R$, there are $e,f \in \text{Id}(R)$ and $j \in J(R)$ such that $1-a = e+f+j$, which yields $a = (1-e)-f-j$.

$(4) \Rightarrow (5)$: For any $a \in R$, there are $e,f \in \text{Id}(R)$ and $j \in J(R)$ such that $$a = e-f+j = e(1-f)-f(1-e)+j,$$ where $e(1-f)$ and $f(1-e)$ are orthogonal idempotents.

$(5) \Rightarrow (1)$: For any $a \in R$, there are $e,f \in \text{Id}(R)$ with $ef=0=fe$ and $j \in J(R)$ such that $a = e-f+j$. Putting $g:=e-f$, it is clear that $g^3=g$. Therefore, $R$ is a DT ring.
\end{proof}

For our next citation, mimicking \cite{cui}, we just remember that a ring $R$ is said to be {\it 2-UJ}, provided the square of each unit is a sum of an idempotent and an element from the Jacobson radical.

\begin{theorem}\label{maj3}
Let $R$ be a ring. Then, the following three conditions are equivalent:

(1) $R$ is a DT ring.

(2) $R$ is a semi-tripotent ring.

(3) For every $a \in R$, there exist $e \in \operatorname{Id}(R)$ and $j \in J(R)$ such that $a^2 = e + j$.
\end{theorem}

\begin{proof}
$(1) \Leftrightarrow (2)$: It follows from Theorem \ref{avali T GDT}.

$(2) \Rightarrow  (3)$: It is evident.

$(3) \Rightarrow  (2)$: It suffices to show that, for every $a \in R$, $a - a^3 \in J(R)$ and idempotents lift modulo $J(R)$. Certainly, for each $a \in R$, we have $a^2 - a^4 \in J(R)$.

Setting $\overline{R} := R/J(R)$, for every $x \in \overline{R}$, we see that $x^2 = x^4$, so that $\overline{R}$ is a strongly $\pi$-regular ring. Now, \cite[Theorem 1]{nstr} tells us that $\overline{R}$ is strongly clean.

Moreover, for each $u \in U(R)$, we write $u^2 - u^4 \in J(R)$. Since $u \in U(R)$, it follows that $1 - u^2 \in J(R)$, which shows that $R$ is a 2-UJ ring. But, via \cite[Lemma 2.5]{cui}, $\overline{R}$ is also a 2-UJ ring, and thus again from \cite[Theorem 3.3]{cui}, $\overline{R}$ is semi-tripotent. This means that, for every $a \in R$, $a - a^3 \in J(R)$.

Now, suppose $a - a^2 \in J(R)$. By assumption, there is $e \in \operatorname{Id}(R)$ such that $a^2 - e \in J(R)$, which insures that
\[
a - e = (a - a^2) - (a^2 - e) \in J(R).
\]
Consequently, idempotents lift modulo $J(R)$, completing the arguments.
\end{proof}

\begin{remark}
The previous theorem quite naturally raises the question: if, for any $a \in R$, there exists $e \in \operatorname{Tr}(R)$ such that $a^3 - e \in J(R)$, is $R$ necessarily a DT ring? Unfortunately, the answer is negative, because a plain verification gives that $\mathbb{Z}_7$ satisfies this property, but however it is {\it not} a DT ring.

Analogically, what can be said about $R$ if, for any $a \in R$, there exists $e \in \operatorname{Id}(R)$ such that $a^4 - e \in J(R)$? Again, the answer is negative, as a plain check shows that $\mathbb{Z}_5$ satisfies this condition, but it is {\it not} a DT ring.
\end{remark}

In this aspect, it was shown in \cite[Theorem 3.5]{kyzr} that a ring $R$ is semi-tripotent if, and only if, for each $a \in R$, $a = ev + j$, where $j \in J(R)$, $e^2 = e \in R$ and $v^2 = 1$ if, and only if, for each $a \in R$, $a = ve + j$, where $j \in J(R)$, $e^2 = e \in R$ and $v^2 = 1$.

\medskip

This quite logically poses the question: if we replace the product with a sum in the above relations, does $R$ remain semi-tripotent? In the next main theorem, we attempt to answer this question in the affirmative. 

\medskip

And so, we close our work with the following.

\begin{theorem}\label{maj4}
Let $R$ be a ring. Then, the following three conditions are equivalent:

(1) $R$ is a DT ring.

(2) $R$ is a semi-tripotent ring.

(3) For each $a \in R$, $a = e+v+j$, where $j \in J(R)$, $e^2 = e \in R$, $v^2 = 1$ and $ev = ve$.
\end{theorem}

\begin{proof}
$(1) \Leftrightarrow (2)$: It follows directly from Theorem \ref{avali T GDT}.

$(2) \Rightarrow  (3)$: This follows at once since, for every $x \in Tr(R)$, we can write $x = (1-x^2) + (x-(1-x^2))$, where $1-x^2 \in Id(R)$ and $(x-(1-x^2))^2=1$.

$(3) \Rightarrow (2)$: Since $U(R)+J(R) = U(R)$, it follows that $R$ is a clean ring, and thus idempotents lift modulo $J(R)$. It remains only to prove that, for every $a \in R$, $a-a^3 \in J(R)$. Indeed, we first show $6 \in J(R)$. Writing $3 = e+v+j$, where $j \in J(R)$, $e^2 = e \in R$ and $v^2 = 1$, we then have that
\[
9 \equiv e+2ev+1 \equiv (3-v)+2(3-v)v+1 \equiv 1+5v \pmod{J(R)},
\]
which guarantees $7 \equiv 5v \pmod{J(R)}$ and, consequently, $24 \in J(R)$, whence $6 \in J(R)$.

Set $\overline{R} := R/J(R)$. Now, by the Chinese Remainder Theorem, one may writes that $\overline{R} \cong R_1 \times R_2$ where $R_1 := \overline{R}/2\overline{R}$ and $R_2 := \overline{R}/3\overline{R}$.

Consider $R_1$, where $2=0$: For $u \in U(R_1)$, by assumption $u = e+v$ with $e^2=e$ and $v^2=1$. Then, $$u^2 = 1+e=1-e \in U(R_1) \cap \operatorname{Id}(R_1) = \{1\},$$ showing $R_1$ is a 2-UJ ring.

Consider $R_2$, where $3=0$: For $u \in U(R_2)$, $u = e+v$ implies $u^3 = e+v = u$, hence $u^2=1$, so $R_2$ is too a 2-UJ ring.

Thus, $\overline{R}$ is simultaneously 2-UJ and clean, and bearing in mind \cite[Theorem 3.3]{cui}, we are able to conclude that $\overline{R}$ is semi-tripotent. Therefore, for every $a \in R$, $a-a^3 \in J(R)$, as needed.
\end{proof}


\vskip1.0pc

\begin{thebibliography}{99}

\bibitem{csj}
H. Chen, On strongly J-clean rings, \textit{Commun. Algebra} {\bf 38} (2010), 3790--3804.

\bibitem{chenb}
H. Chen, Rings Related Stable Range Conditions, \textit{Series in Algebra} {\bf 11}, World Scientific, Hackensack, NJ, 2011.

\bibitem{css}
H. Chen and M. Sheibani, Strongly 2-nil-clean rings, \textit{J. Algebra Appl.} {\bf 16}(8) (2017).

\bibitem{chensjc}
H. Chen, Strongly J-clean matrices over local rings, \textit{Commun. Algebra} {\bf 40}(4) (2012), 1352--1362.

\bibitem{coon}
I. G. Connell, On the group ring, \textit{Can. J. Math.} {\bf 15} (1963), 650--685.

\bibitem{cui}
J. Cui and X. Yin, Rings with 2-UJ property, \textit{Commun. Algebra} {\bf 48}(4) (2020), 1382--1391.

\bibitem{dl}
P. V. Danchev and T. Y. Lam, Rings with unipotent units, \textit{Publ. Math. Debrecen} {\bf 88}(3-4) (2016), 449--466.

\bibitem{diesl}
A. J. Diesl, Nil clean rings, \textit{J. Algebra} {\bf 383} (2013), 197--211.

\bibitem{kkqt}
F. Karabacak, M. T. Ko\c{s}an, T. Quynh and D. Tai, A generalization of UJ-rings, \textit{J. Algebra Appl.} {\bf 20} (2021).

\bibitem{kwz}
M. T. Ko\c{s}an, Z. Wang and Y. Zhou, {\it Nil-clean and strongly nil-clean rings}, J. Pure Appl. Algebra {\bf 220}(2) (2016), 633--646.

\bibitem{kyzr}
M. T. Ko\c{s}an, T. Yildirim and Y. Zhou, Rings whose elements are the sum of a tripotent and an element from the Jacobson radical, \textit{Can. Math. Bull.} {\bf 62}(4) (2019), 810--821.

\bibitem{kzw}
M. T. Ko\c{s}an and Y. Zhou, On weakly nil-clean rings, \textit{Front. Math. China} {\bf 11} (2016), 949--955.

\bibitem{lamf}
T. Y. Lam, A First Course in Noncommutative Rings, \textit{Graduate Texts in Mathematics} {\bf 131}, Springer, Berlin, Heidelberg, New York, 2001.

\bibitem{lame}
T. Y. Lam, Exercises in Classical Ring Theory, \textit{Problem Books in Mathematics}, Springer, Berlin, Heidelberg, New York, 2003.

\bibitem{lmr}
A. Leroy and J. Matczuk, Remarks on the Jacobson radical, \textit{Contemp. Math.}, in: Rings, Modules and Codes {\bf 727} (2019), 269--276.

\bibitem{nstr}
W. K. Nicholson, Strongly clean rings and Fitting's lemma, \textit{Commun. Algebra} {\bf 27} (1999), 3583--3592.

\bibitem{nzu}
W. K. Nicholson and Y. Zhou, Rings in which elements are uniquely the sum of an idempotent and a unit, \textit{Glasg. Math. J.} {\bf 46} (2004), 227--236.

\bibitem{zhoucl}
Y. Zhou, On clean group rings, Advances in Ring Theory, \textit{Treads in Mathematics}, Birkhauser, Verlag Basel/Switzerland, 2010, pp. 335--345.

\end{thebibliography}
\end{document}